\documentclass{amsart}
\usepackage{amsmath,amsfonts,amssymb,mathtools,amsthm,dsfont,bm}
\usepackage[colorlinks,linkcolor=blue,citecolor=blue]{hyperref}
\usepackage{tikz}
\usetikzlibrary{cd, calc, knots}
\usepackage{xcolor}

\usepackage[utf8]{inputenc}
\usepackage{amsbsy}

\newcommand{\al}{\alpha}

\def\b{\beta}

\newcommand{\De}{\Delta}

\newcommand{\la}{\lambda}

\newcommand{\z}{\zeta}

\newcommand{\te}{\theta}

\newcommand{\x}{\xi}

\newcommand{\sm}{\sigma}

\newcommand{\f}{\phi}
\newcommand{\F}{\Phi}

\newcommand{\om}{\omega}

\newcommand{\ol}[1]{\overline{#1}}
\newcommand{\op}{\oplus}
\newcommand{\ot}{\otimes}
\newcommand{\bt}{\boxtimes}

\renewcommand{\o}{\otimes}

\newcommand{\rational}{\mathbb{Q}}

\def\Vec{\operatorname{Vec}}
\def\Fun{\operatorname{Fun}}

\newcommand{\fg}{\mathfrak{g}}

\newcommand{\fA}{\mathfrak{A}}
\newcommand{\fd}{\mathfrak{d}}

\newcommand{\so}{\mathfrak{so}}

\newcommand{\BB}{\mathcal{B}}
\newcommand{\BC}{\mathbb{C}}
\newcommand{\BN}{\mathbb{N}}
\newcommand{\BQ}{\mathbb{Q}}
\newcommand{\BZ}{\mathbb{Z}}

\newcommand{\BR}{\mathbb{R}}
\renewcommand{\AA}{\mathcal{A}}
\newcommand{\CC}{{\mathcal{C}}}
\newcommand{\DD}{\mathcal{D}}
\newcommand{\II}{\mathcal{I}}
\newcommand{\HH}{\mathcal{H}}
\newcommand{\LL}{\mathcal{L}}
\newcommand{\ZZ}{\mathcal{Z}}
\newcommand{\EE}{\mathcal{E}}
\def\FF{\mathcal{F}}
\newcommand{\RR}{\mathcal{R}}
\newcommand{\WW}{\mathcal{W}}

\def\UU{\mathcal{U}}

\newcommand{\ld}{\lambda}

\newcommand{\e}{\varepsilon}

\newtheorem{thm}{Theorem}[section]
\newtheorem{lem}[thm]{Lemma}
\newtheorem{prop}[thm]{Proposition}
\newtheorem{cor}[thm]{Corollary}
\newtheorem{rmk}[thm]{Remark}

\theoremstyle{definition}

\newcommand{\ev}{\operatorname{ev}}
\newcommand{\coev}{\operatorname{coev}}

\def\sl{\mathfrak{sl}}

\newcommand{\id}{\operatorname{id}}

\newcommand{\gal}{\operatorname{Gal}}
\newcommand{\irr}{\operatorname{Irr}}
\newcommand{\FPdim}{\operatorname{FPdim}}

\newcommand{\ord}{\operatorname{ord}}

\newcommand{\Rep}{\operatorname{Rep}}

\newcommand{\sgn}{\operatorname{sgn}}

\newcommand{\lcm}{\operatorname{lcm}}

\newcommand{\1}{\mathds{1}}

\def\GQ{\gal(\bar\BQ/\BQ)}

\def\pt{{\operatorname{pt}}}

\def\O2{\Omega_2}

\def\Sem{\mbox{\upshape \textbf{Sem}}}

\def\f2{\varphi_{2}}

\newcommand{\NN}{\mathcal{N}}

\newcommand{\jac}[2]{\left( \frac{#1}{#2} \right)_{\operatorname{J}}}

\renewcommand{\AA}{\mathcal{A}}

\def\Vs{{\operatorname{Vec}}}

\def\fQ{{\mathfrak{Q}}}
\def\fP{{\mathfrak{P}}}

\usepackage{fullpage}

\title{The Witt classes of $\so(2r)_{2r}$}
\author[Rowell]{Eric C. Rowell}
\address{Department of Mathematics\\
    Texas A\&M University\\
    College Station, TX 77843-3368\\
    U.S.A.}
    \email{rowell@math.tamu.edu}
\author[Ruan]{Yuze Ruan}
\address{Beijing Institute of Mathematical Science and Applications (BIMSA)\\
Huairou, Beijing, China.}
    \email{yuzeruan@bimsa.cn}
\author[Wang]{Yilong Wang}
\address{Beijing Institute of Mathematical Science and Applications (BIMSA)\\
Huairou, Beijing, China.}
    \email{wyl@bimsa.cn}
\date{\today}

\begin{document}

\maketitle
\begin{abstract}
We study the Witt classes of the modular categories $\so(2r)_{2r}$ associated with quantum groups of type $D_r$ at $(4r-2)$-th roots of unity. From these classes we derive infinitely many Witt classes of order 2 that are linearly independent modulo the subgroup generated by the pointed modular categories.  In particular we produce an example of a simple,  completely anisotropic modular category that is not pointed whose Witt class has order 2, answering a question of Davydov, M\"uger, Nikshych and Ostrik.  Our results show that the trivial Witt class $[\Vec]$ has infinitely many square roots modulo the pointed classes, in analogy with the recent construction of infinitely many square roots of the Ising Witt classes modulo the pointed classes constructed in a similar way from certain type $B_r$ modular categories.  We compare the subgroups generated by the Ising square roots and $[\Vec]$ square roots and provide evidence that they also generate linearly independent subgroups.
\end{abstract}

\section{Introduction}
Modular categories play a central role in the study of 2-dimensional topological phases of matter \cite{Nayaketal2008}, invariants of 3-manifolds \cite{Tur10} and conformal field theory \cite{MooreSeiberg}.  The problem of classifying modular categories led to the definition \cite{DMNO} of 
the Witt group $\WW$ for non-degenerate braided fusion categories, generalizing the Witt group for abelian groups equipped with non-degenerate quadratic forms (see \cite{DGNO}). Two non-degenerate braided fusion categories $\CC$ and $\DD$ are Witt equivalent if they are equivalent ``modulo Drinfeld centers," i.e., if $\CC \boxtimes \ZZ(\AA) \cong \DD\boxtimes \ZZ(\BB)$ for some fusion categories $\AA$, $\BB$.

 Recently \cite{NRWZ20a} it was shown that the 8 Witt classes of Ising categories have infinitely many independent square roots modulo the subgroup generated by pointed categories.  This was achieved using the \emph{signature}, a new Witt class invariant related to the higher central charge introduced in \cite{NSW19}.  From this a verification of a conjecture of \cite{DNO} was derived, namely, that the torsion subgroup $s\WW_2$ of the super-Witt group has infinite rank.  The categories representing the square roots of the Ising Witt classes are of the form $\so(N)_N$ with $N$ odd, obtained from the (Lie type $B$) quantum groups  $U_q\so(N)$ at $q=e^{\pi i/(4N-4)}$.  In this article we study the case $N$ even, obtained from (Lie type $D$) quantum groups of  $U_q\so(N)$ with $q=e^{\pi i/(2N-2)}$.  From these we obtain infinitely many Witt classes of order $2$, which intersect trivially with the subgroup generated by the pointed and Ising classes.  Thus we have infinitely many square roots of the trivial Witt class $[\Vec]$, modulo the subgroup generated by the pointed and Ising Witt classes.  From this we derive an affirmative answer to  \cite[Question 6.8]{DMNO}: there \emph{does} exist a completely anisotropic simple modular category whose Witt class has order 2.  We also compare the Lie type $D$ Witt classes studied here with the Lie type $B$ Witt classes of \cite{NRWZ20a}.  Although we cannot show there are no relations among them, we do find an infinite sequence of type $D$ classes that intersects trivially with the infinite sequence of type $B$ Witt classes used in \cite{NRWZ20a} to verify the conjecture of \cite{DNO}.

Our results further illustrate the usefulness of the signature for distinguishing Witt classes.  The methods we use are  number-theoretical and thus require carefully chosen parameters.  It would be interesting to understand the relations among all of these Witt classes.  We believe that the examples we study provide infinitely many Witt-independent examples of non-pointed completely anisotropic simple modular categories (so-called property $S$ categories) with order 2 Witt classes, but we only verify a single example.  

Here is a more detailed outline of this article.  After providing the notation and context of the problem we derive square roots of $[\Vec]$ from quantum groups of Lie type $D_r$, treating the $r$ odd and $r$ even case separately.  Next we compute their signatures, and use them to find a sub-sequence of these Witt classes that are linearly independent, modulo the subgroup generated by pointed and Ising categories.  Next we compare the Witt subgroups we study with those of \cite{NRWZ20a}.  Finally we show that at least one square root of $[\Vec]$ is both simple and anisotropic, answering a question of \cite{DMNO}.

\section*{Acknowledgement}
The authors would like to thank Siu-Hung Ng and Andrew Schopieray for helpful discussions, especially for pointing out that the total positivity of connected \'etale algebras can be used to simplify arguments in Section 6. We also thanks Daniel Bump for the very useful \texttt{FusionRing} package in \texttt{SageMath}, used to test our results. E.C.R. was partially supported by NSF grant DMS-1664359 and a Presidential Impact Fellowship from Texas A\&M University.

\section{Preliminaries and Notation}

We assume some familiarity with the standard notions in the theory of fusion categories, referring the reader to \cite{EGNO} for a complete treatment. Mainly to fix notation we provide some details.
\subsection{Fusion categories}\label{subsec:FC}
A \emph{fusion category} over $\BC$ is a semisimple, $\BC$-linear abelian, rigid monoidal category with finite-dimensional Hom-spaces and finitely many isomorphism classes of simple objects among which is the tensor unit $\1$ \cite{ENO05}. For any fusion category $\CC$, we denote by $\irr(\CC)$ a complete set of representatives of the isomorphism classes of $\CC$. The tensor product endows $K_0(\CC)$, the Grothendieck group of $\CC$, with a ring structure. More precisely, we have $X\ot Y = \sum_{Z\in\irr(\CC)} N_{X,Y}^{Z} Z$ for any $X, Y\in\irr(\CC)$, where
$$N_{X,Y}^{Z} := \dim_{\BC}\CC(X\ot Y, Z)\,.$$ 
For any $X\in\irr(\CC)$, let $\NN_{X}$ be the square matrix of size $|\irr(\CC)|$ such that $(\NN_{X})_{Y, Z} = N_{X,Y}^{Z}$ for any $Y, Z \in \irr(\CC)$. The \emph{Frobenius-Perron dimension} of $X\in\irr(\CC)$, denoted by $\FPdim(X)$, is the largest positive eigenvalue for $\NN_{X}$ (see \cite{ENO05}). The Frobenius-Perron dimension of $\CC$ is defined to be
$$
\FPdim(\CC) = \sum_{X \in \irr(\CC)} \FPdim(X)^2\,.
$$

Let $\CC$ be a fusion category. The rigidity of $\CC$ means that for any object $X\in\CC$, there is a left dual $(X^*, \ev_X, \coev_X)$, where $X^*\in\CC$ is an object, $\ev_{X}: X^* \o X \to \1$ and $\coev_{X}: \1 \to X \o X^{*}$ are morphisms satisfying the duality conditions in \cite[Def.~2.10.1]{EGNO}. The notion of a right dual is similarly defined \cite[Def.~2.10.2]{EGNO}. It is well-known that dual objects are unique up to isomorphism (see, for example, \cite[Lem.~2.1.5]{BK01}). A simple object $X$ of $\CC$ is called \emph{invertible} if $\ev_X$ and $\coev_X$ above are isomorphisms. A fusion category is \emph{pointed} if all of its simple objects are invertible. We denote the maximal pointed fusion subcategory of a fusion category $\CC$ by $\CC_\pt$. A fusion category $\CC$ is called \emph{unpointed} if $\irr(\CC_\pt) = \{\1\}$.

The \emph{global dimension} $\dim(\CC)$ of a fusion category $\CC$ was introduced in \cite[Def.~2.5]{Mug031}. By \cite[Rmk.~2.5]{ENO05}, $\dim(\CC)$ is a totally positive algebraic integer. We will denote the positive square root of $\dim(\CC)$ by $\sqrt{\dim(\CC)}$, which will play a central role in our results. Finally, a fusion category $\CC$ is called \emph{pseudounitary} if $\dim(\CC) = \FPdim(\CC)$. 

\subsection{Braided and ribbon fusion categories}\label{subsec:BFC}
Let $\CC$ be a fusion category. A \emph{braiding} on $\CC$ is a natural isomorphism 
$\b_{V, W}: V \ot W \xrightarrow{\cong} W \ot V$ 
satisfying the Hexagon axioms (see \cite[Chap.~8]{EGNO}). A fusion category equipped with a braiding is called a \emph{braided fusion category}. 

Let $\CC$ be a braided fusion category equipped with a braiding $\beta$. For any fusion subcategory $\DD$ in $\CC$, the \emph{centralizer} of $\DD$ in $\CC$, denoted by $\DD'$, is the full fusion subcategory of $\CC$ generated by objects $V \in \CC$ such that $\b_{W, V}\circ\b_{V, W} = \id_{V \ot W}$ for all $W \in \DD$. The centralizer $\CC'$ of $\CC$ itself is called the \emph{M\"uger center} of $\CC$.  A braided fusion category $\CC$ is called \emph{non-degenerate} if $\irr(\CC') = \{\1\}$, i.e., $\CC'$ is equivalent to $\Vs$, the category of finite-dimensional vector spaces over $\BC$. 

For any fusion category $\CC$, its \emph{Drinfeld center}, denoted by $\ZZ(\CC)$, is the category with objects of the form $(V, c_{-,V})$, where $V \in \CC$ and $c_{X, V}: X \ot V \xrightarrow{\cong} V\ot X$ is a natural family of isomorphisms satisfying the half-braiding conditions in \cite[Def.~XIII.4.1]{Kas95}. It is well-known that $\ZZ(\CC)$ is a non-degenerate braided fusion category (see, for example \cite[Cor.~3.9]{DGNO}).

A braided fusion category $\CC$ is called \emph{symmetric} if $\CC' = \CC$. 
The fusion category $\Rep(G)$ of finite-dimensional complex representations of a finite group $G$, endowed with the standard braiding, is a symmetric fusion category also denoted by $\Rep(G)$ for convenience. A symmetric fusion category equivalent to $\Rep(G)$ for some finite group $G$ is called \emph{Tannakian}.

A \emph{ribbon structure} on a braided fusion category $\CC$ equipped with a braiding $\beta$ is a natural isomorphism $\theta: \id_\CC \to \id_\CC$ of the identity functor satisfying $\theta_{V^{*}} = \theta_V^*$ and
\begin{equation}\label{eq:twist-def}
\theta_{V\ot W} = (\theta_V \ot \theta_W)\circ \b_{W, V} \circ \b_{V, W} 
\end{equation}
for any $V, W \in \CC$. In particular, for any $X \in \irr(\CC)$, $\theta_X$ is equal to a non-zero scalar times $\id_X$. By an abuse of notation, we denote both the scalar and the isomorphism itself by $\theta_X$ for all simple $X$, and we call $\theta_X$ the \emph{(topological) twist} of $X$. 

A braided fusion category $\CC$ with a ribbon structure $\theta$ is called a \emph{ribbon fusion category}, or a \emph{premodular category}. In a premodular category, endomorphisms are equipped with a canonical trace valued in $\BC$. In particular, this leads to well-defined \emph{quantum dimensions} $d_V:=\dim_\CC(V)$ for objects $V\in\CC$, and one has $\dim(\CC)=\sum_{V\in\irr(\CC)}d_V^2$. For details, see \cite{BK01,ENO05}. A premodular category $\CC$ is called \emph{modular} if the underlying braided fusion category is non-degenerate. 

The \emph{T-matrix} of a premodular category $\CC$ is defined to be the diagonal matrix
$$
T_{X,Y} := \theta_X\cdot \delta_{X, Y}\,,\,\quad X,Y\in\irr(\CC).
$$
It is shown by Vafa \cite{Vafa88} that if $\CC$ is a modular category, then the order of its T-matrix, denoted by $\ord(T_\CC)$, is finite. 

Numerical invariants of modular categories can be obtained from quantum dimensions and twists. For example, for any modular category $\CC$ and any integer $n\in\BZ$, the $n$-th \emph{Gauss sum} is defined in \cite{NSW19} as $\tau_n(\CC) = \sum_{X \in \irr(\CC)} d_X^2\theta_X^n$. When $\tau_n \ne 0$, the $n$-th \emph{central charge} is defined to be
\begin{equation}\label{eq:xi-n-def}
\xi_n(\CC) := \frac{\tau_n(\CC)}{|\tau_n(\CC)|}\,,
\end{equation}
where the denominator on the right hand side represents the (complex) absolute value of $\tau_n(\CC)$.

By \cite[Prop.~2.10]{ENO05} and the finiteness of $\ord(T_\CC)$, the complex conjugate of $\tau_n(\CC)$ equals to $\tau_{-n}$ for all $n \in \BZ$. In particular, when $n = 1$, we can also write
\begin{equation}\label{eq:xi-1-general}
\xi_1(\CC) = \frac{\tau_1(\CC)}{\sqrt{\dim(\CC)}}\,.
\end{equation}
These invariants provide insights on modular categories and the Witt group, as is demonstrated in \cite{NSW19, NRWZ20a}.

Finally, we will need the following classification result on pointed braided fusion categories (for details, see \cite[Sec.~8.4]{EGNO}): such categories are parameterized by pairs $(A,Q)$ and denoted by $\CC(A, Q)$, where $A$ is a finite abelian group and $Q:A \to \BC^\times$ is a quadratic form on $A$. Simple objects in $\CC(A, Q)$ are in one-to-one correspondence to the elements of $A$. Moreover, $\CC(A, Q)$ is a ribbon fusion category whose ribbon structure is determined by the twists $\theta_a = Q(a)$ for $a \in A = \irr(\CC(\AA, Q))$, and it is modular if and only if the map $(a,b)\mapsto \frac{Q(a+b)}{Q(a)Q(b)}$ is a non-degenerate bicharacter on $A$.

An important example of pointed modular categories is the \emph{semion} modular category \cite{RSW09}, denoted by $\Sem$. Using the notation above, we characterize $\Sem$ as $\CC(\BZ/2\BZ, Q)$ with $Q(a)=i^{a^2}$ for $a\in\BZ/2\BZ$. It has two invertible objects with the non-unit object having dimension 1 and twist $i$. In particular, 
\begin{equation}\label{eq:xi-1-sem}
\xi_1(\Sem) = \exp\left(\frac{\pi i}{4}\right)\,.    
\end{equation}

\subsection{Witt group \texorpdfstring{$\WW$}{}}\label{subsec:prelim-Witt}
The concept of Witt equivalence of non-degenerate braided fusion categories is introduced in \cite{DMNO}. More precisely, two non-degenerate braided fusion categories $\CC$ and $\DD$ are Witt equivalent if there exist fusion categories $\AA$ and $\BB$ such that $\CC \boxtimes \ZZ(\AA) \cong \DD \boxtimes \ZZ(\BB)$. We denote the Witt equivalence class of a category $\CC$ by $[\CC]$. Under the Deligne product, Witt equivalence classes form an abelian group denoted by $\WW$. The Witt group $\WW$ of non-degenerate braided fusion categories can be viewed as a generalization of the classical Witt group of non-degenerate quadratic forms on abelian groups (i.e., metric groups), which generate the pointed part of $\WW$ denoted by $\WW_\pt$. The structure of $\WW_\pt$ is well known \cite{DGNO,DMNO}. The study of the Witt group is closely related to the theory of conformal embeddings in the context of affine Kac-Moody algebras \cite{BB,SW} and to phase transitions driven by anyon condensation and symmetry gauging in topological phases of matter \cite{Burnell,BJLP}. 

Numerical invariants are powerful tools in the study of the structure of the Witt group. For example, the first central charge $\xi_1$ (see \eqref{eq:xi-n-def}) is used in the study of conformal embeddings of rational vertex operator algebras associated to affine Lie algebras \cite{DMNO}. Another important Witt invariant, namely, the \emph{signature homomorphism}, is introduced in \cite{NRWZ20a}. Here, we briefly recall the definition as follows. For any fusion category $\CC$, it is well-known that $\FPdim(\CC)$ is totally positive algebraic integer \cite{ENO05}. Therefore, its positive square root, denoted by $\sqrt{\FPdim(\CC)}$, is a totally real algebraic integer, and the signature of $\CC$ is then defined to be the sign of its Galois conjugates. More precisely, let $\GQ$ be the absolute Galois group of $\BQ$, then the signature of $\CC$ is the function
\begin{equation}\label{eq:sgn}
\e_\CC: \GQ \to \{\pm 1\}\,,
\quad
\sigma \mapsto \sgn\left(\sigma(\sqrt{\FPdim(\CC)})\right)\,.
\end{equation}
By Galois theory, it is easy to see that $\e_\CC$ factors through any Galois extension of $\BQ$ containing $\sqrt{\FPdim(\CC)}$. Therefore, in the sections below, we will use large enough fields (instead of $\Bar{\BQ}$) to simplify computations. Let $\UU_2 := \{\pm 1\}^{\GQ}$ be the group of functions from $\GQ$ to $\{\pm 1\}$. Then by \cite[Thm.~3.4]{NRWZ20a}, the signature map induces a well-defined group homomorphism 
\begin{equation}
I: \WW \to \UU_2\,,
\quad
[\CC] \mapsto \e_\CC\,,
\end{equation}
which is the main technical tool in this paper.

It is natural to pursue constructions of new non-degenerate braided fusion categories that are Witt equivalent to a given one, and taking local modules of connected \'etale algebras is one such construction. Let $\CC$ be a braided fusion category with braiding $\beta$. The notion of an \emph{connected \'etale algebra} object $A\in \CC$ is defined in \cite[Sec.~3.1]{DMNO}. It is shown in \cite{DMNO} that if $A \in \CC$ is a connected \'etale algebra, then $\CC_A$, the category of right $A$-modules, is a fusion category. Moreover, according to \cite{Par95,KO02}, $\beta$ descends to $A$-module maps for \emph{local} modules. Here a (right) $A$-module $M$ with $A$-action $\alpha: M \ot A \to M$ is called \emph{local} (or \emph{dyslectic}) if it satisfies $\alpha \circ \beta_{A, M} \circ \beta_{M, A} = \alpha$. As an immediate consequence, $\CC_A^0$, the full subcategory of $\CC_A$ of local modules of $A$, is a braided fusion category. In \cite[Sec.~3]{DMNO}, it is pointed out that if $\CC$ is a ribbon fusion category with ribbon structure $\theta$, then a rigid algebra $A \in \CC$ in the sense of \cite{KO02} is connected \'etale with non-zero quantum dimension if and only if  $\theta_A = \id$. Group symmetry gives rise to examples of the local module construction. More precisely, when $\CC$ contains a Tannakian fusion subcategory $\Rep(G)$ for some finite group $G$, then the regular algebra $A=\Fun(G)$ is a connected \'etale algebra in $\CC$. The corresponding local module $\CC_A^0$ can also be characterized in the context of de-equivariantization \cite{DGNO, DMNO}.

The importance of the local module construction is well illustrated in the combination of Corollary 3.32 and Proposition 5.4 of \cite{DMNO}: for any non-degenerate braided fusion category $\CC$ and any connected \'etale algebra $A \in \CC$, we have
\begin{equation}
\FPdim(\CC_A^0) = \frac{\FPdim(\CC)}{\FPdim(A)^2}
\quad\text{and}\quad [\CC] = [\CC_A^0]\,.
\end{equation}
Therefore, in order to study the Witt class of $\CC$, it suffices to study the smaller category $\CC_A^0$. Moreover, the local module construction is closely related to the conformal embedding in the theory of rational vertex operator algebras, see \cite{DMNO}. 

By taking the local module category of a maximal connected \'etale algebra in a non-degenerate braided fusion category $\CC$, one gets a ``minimal'' representative in its Witt class in the following sense. A non-degenerate braided fusion category is called \emph{completely anisotropic} if it does not contain any connected \'etale algebra other than $\1$. It is shown in \cite[Thm.~5.13]{DMNO} that any Witt class contains a unique (up to braided equivalence) completely anisotropic representative. While specific Witt classes can be difficult to compare in practice, having a completely anisotropic representative is particularly useful. A fusion category $\CC$ is called \emph{simple} if $\CC$ has no non-trivial fusion subcategories. There are examples of completely anisotropic categories which are not simple and vice versa. It is then interesting to consider the categories which are ``minimal'' in both senses, which leads to the concept of \emph{property S} \cite[Sec.~5.4]{DMNO}. A non-degenerate braided fusion category $\CC$ has property S if it is completely anisotropic, simple and unpointed. The subgroup of $\WW$ generated by the categories with property $S$ is denoted by $\WW_{S}$.

\section{Categorical data and local modules of \texorpdfstring{$\so(2r)_{2r}$}{}}
In this section, we set up notations and provide the categorical data. We will also discuss basic properties of the local modules of the quantum group modular categories $\so(2r)_{2r}$.

\subsection{Categorical data of \texorpdfstring{$\so(2r)_{2r}$}{}}\label{sec:lie}
Quantum groups give prominent examples of modular categories whose origin traces back to the dawn of the subject \cite{Dri87, Lus93}. For any simple Lie algebra $\fg$ over $\BC$ and a positive integer $k$ called the \emph{level}, one gets a modular category $\fg_k$ by taking the semisimplification of the tilting module category of the quantum group $U_q(\fg)$ specialized at a root of unity $q$ determined by $\fg$ and $k$. The reader is referred to \cite{BK01, Row05f} for details. 

Let $r\geq 2$ be an integer and $\CC_{r}$ be the quantum group modular category $\so(2r)_{2r}$. In this section, we study the structure of the Witt subgroup generated by these categories using signature homomorphism (see \eqref{eq:sgn}). Necessary data and notation is given below. For more details, see, for example, \cite{Row05f, BK01, Hum72}.

\begin{itemize}
\item 
Orthonormal basis for the inner product space $(\BR^r, (\cdot\mid\cdot))$: $\{e_1, \ldots, e_r\}$.

\item 
The set of positive roots: $\De_+ = \{e_j \pm e_k \mid 1\leq j<k\leq r\}$. Root lattice: $\fQ$. Coroot lattice: $ \fQ^{\vee}$.

\item 
Half sum of positive roots: $\rho=(r-1)e_1+(r-2)e_2\cdots+e_{r-1}$. 

\item 
Fundamental weights: 
$\om_j=\sum_{i=1}^{j}e_i$, $1\leq j\leq r-2$;
$\om_{r-1}=\frac{1}{2}\sum_{i=1}^{r-1}e_i-\frac{1}{2}e_r$, $\om_{r}=\frac{1}{2}\sum_{i=1}^{r}e_i$.

\item
Dual Coxeter number: $h^{\vee}=2r-2$. The set of dominant weights: $\F_+$. Weight lattice: $\fP$. 

\item
The fundamental alcove $\fA_r$ is in one-to-one correspondence with the set of simple objects of $\CC_r$.
\begin{equation}
\label{eq:fund-alc}
\begin{aligned}
\fA_r
=\{\lambda\in \F_+ : (\lambda|e_1+e_2)\leq 2r\}=\irr(\CC_r)\,.
\end{aligned}
\end{equation}

\item
Quantum parameter: $q=\exp\left(\frac{\pi i}{(2r+h^{\vee})}\right)=\exp\left(\frac{\pi i}{4r-2}\right)$. Twist of simple objects:
\begin{equation}\label{eq:q-dim}
\te_{\la}=q^{(\la|\la+2\rho)}\,,\quad\forall\la\in \fA_r\,.
\end{equation}

\item
First central charge:
\begin{equation}\label{eq:xi-1}
    \x_1(\CC_{r})=\exp\left(\frac{\pi i r^2}{4}\right)\,.
\end{equation}
\end{itemize}

In Section \ref{sec:4}, we will need the following auxiliary function. We adopt the standard notation and write the floor and ceiling functions as $x\mapsto\lfloor x\rfloor, x\mapsto\lceil x\rceil$ respectively.

\begin{lem}\label{lem:fdrj}
For any integers $r \ge 1$ and $j \ge 1$, define $\fd_r(j) := |\{\al\in \De_{+}:\ (\al|\rho)=j\}|$. Then
\begin{equation} \label{eq:m} 
\mathfrak{d}_r(j) =
\begin{cases}
    r-\lfloor\frac{j}{2}\rfloor &  \text{ if }1\leq j\leq r-1\,, \\
    r-\lceil\frac{j+1}{2}\rceil &  \text{ if }r\leq j \leq 2r-3\,, \\
    0 & \text{ otherwise}\,.
\end{cases}
\end{equation}
\end{lem}
\begin{proof}
Let $\alpha\in\Delta_+$ be a positive root, there are two cases. If $\al=e_a+e_b$ for some $1\leq a<b\leq r$, $(\al|\rho)=2r-(a+b)$, then $1\leq(\al|\rho)\leq 2r-3$; if $\al=e_a-e_b$ for some $1\leq a<b\leq r$, $(\al|\rho)=b-a$, then $1\leq b-a\leq r-1$. Therefore, $\fd_r(j) = 0$ for all $j > 2r-3$.

Define 
\[\begin{split}
&\mathfrak{l}^+_r(i)=|\{(a,b)\in \BZ^2:1\leq a<b\leq r,\ a+b=i\}|\,,\\
&\mathfrak{l}^-_r(i)=|\{(a,b)\in \BZ^2:1\leq a<b\leq r,\ b-a=i\}|\,.
\end{split}\] 
For $1\leq j\leq 2r-3$, we have: 
$$
\mathfrak{d}_r(j)=\mathfrak{l}^+_r(2r-j)+\mathfrak{l}^-_r(j)
$$
It's clear by induction that
\[ 
\mathfrak{l}^+_r(j) =
  \begin{cases}
   \lfloor\frac{j-1}{2}\rfloor\quad & 3\leq j\leq r\,,\\
   \lceil\frac{2r-j}{2}\rceil\quad & r+1\leq j\leq 2r-1\,,    
  \end{cases}
\]
and
\[
\mathfrak{l}^-_r(j)=
\begin{cases}
 r-j\quad & 1\leq j\leq r-1\,,\\
 0\quad & r\leq j\leq 2r-3\,.
\end{cases}
\]

Therefore, if $1 \le j \le r-1$, then $\fd_r(j) = r-j+\lceil\frac{j}{2}\rceil=r-\lfloor\frac{j}{2}\rfloor$; if $r \le j \le 2r-3$, then we have $\fd_r(j) = \lfloor\frac{2r-j-1}{2}\rfloor=r-\lceil\frac{j+1}{2}\rceil$.
\end{proof}

\begin{lem}\label{Lem:odd in m}
For any integer $r\ge 1$, let $S_r=\{1\leq j\leq 2r-3 :\ \fd_r(j) \cdot j\equiv 1\pmod{2} \}$. Then $\mid S_r\mid$ is even.
\end{lem}
\begin{proof}
From \eqref{eq:m}, for any $k \ge 1$ we have
\[
S_r=
\begin{cases}
\{4m+1,4n+3\ :\ 0\leq m\leq k-1,\  k\leq n\leq 2k-1\} & \text{ if } r=4k-3\,,\\
\{4m+3,4n+1\ :\ 0\leq m\leq k-1,\  k+1\leq n\leq 2k\} & \text{ if } r=4k-2\,,\\
\{4m+1,4n+3\ :\ 0\leq m\leq k,\  k\leq n\leq 2k\} & \text{ if } r=4k-1\,,\\
\{4m+3,4n+1\ :\ 0\leq m\leq k-1,\  k\leq n\leq 2k-1\} & \text{ if } r=4k\,.
\end{cases}
\]
In any of the cases above, $\mid S_r\mid$ is even, so we are done.
\end{proof}

For simplicity, we adopt the following conventions. For any positive integer $n$, we set $\zeta_n := \exp\left(\frac{2\pi i}{n}\right)$, and we use $\BQ_n$ to denote the cyclotomic field $\BQ(\zeta_n)$. 
For any $m \in \BN$ and $k$ coprime to $m$, we use $\sm_k$ to denote the element in $\gal(\BQ_m /\BQ)$ sending $\zeta_m$ to $\zeta^k_m$. If $m$ a positive odd integer, and $a \in \BZ$, then we denote the \emph{Jacobi symbol} of $a$ modulo $m$ by $\jac{a}{m}$. 

\begin{lem}\label{lem:sin}
Let $r$ and $u$ be positive integers such that such that $8(2r-1)\mid u$. Then for any $k\in\BZ$ with $\gcd(k,4r-2)=1$, and any $j\in\BZ$, we have
\begin{equation}\label{eq:Sin}
\sm_k\left(\sin\left(\frac{j\pi}{4r-2}\right)\right)=\jac{-1}{k}\sin\left(\frac{kj\pi}{4r-2}\right)\,.
\end{equation}
in the cyclotomic field $\BQ_u$.
\end{lem}
\begin{proof}
The assumption that $8(2r-1)\mid u$ implies both $\zeta_{4r-2}$ and $i$ are in $\mathbb{Q}_{u}$. Since $k$ is odd, so $\sm_k(i)=\jac{-1}{k} i$. Now the lemma follows from $\sin(\frac{j\pi}{4r-2}) = \frac{\zeta^j_{4r-2}-\zeta^{-j}_{4r-2}}{2i}$.
\end{proof}


\begin{lem} \label{lem:field}
Let $T_r$ be the T-matrix of $\CC_r$, and $N_r=\ord(T_r)$. For any positive integer $r \ge 2$, we have $N_r=2^s(2r-1)$ for some integer $0 \le s \le 4$, and $\sqrt{\dim(\CC_r)}\in \rational_{N_r}$.
\end{lem}

\begin{proof}
Since $2e_1,\ \om_r\in \fA_r$ (see \eqref{eq:fund-alc}), by \eqref{eq:q-dim} we have 
$$
\begin{aligned}
&\te_{2e_1}=q^{(2e_1|2e_1+2\rho)}=q^{4r}
=-\exp\left(\frac{\pi i}{2r-1}\right)\,,\\
&\te_{\om_r}=q^{(\om_r|\om_r+2\rho)}=q^{\frac{(2r-1)r}{4}}=\exp\left(\frac{r\pi i}{8}\right)\,.
\end{aligned}
$$
Moreover, since for any $\la\in \fA_r$, $(\la|\la+2\rho)\in \frac{1}{4}\BZ$, $\theta_{\la}=q^{(\la|\la+2\rho)}$ is $16(2r-1)$th-root of unity.
Therefore, if $r$ is even, then $2r-1\mid N_r \mid 16(2r-1)$. When $r$ is odd, we have $16(2r-1)\mid N_r \mid 16(2r-1)$, and so $N_r=16(2r-1)$. This proves the first part of the lemma.

By \cite[Thm.~5.5]{NS07}, $d_\la \in \BQ_{N_r}$ for any $\la \in \fA_r$, which implies $\tau_1(\CC_r) \in \BQ_{N_r}$. Moreover, by \eqref{eq:xi-1}, when $r$ is even, $\x_1(\CC_r)=\pm1\in \BZ\subset\BQ_{N_r}$, and when $r$ is odd, we have $\x_1(\CC_r)=\exp\left(\frac{\pi i}{4}\right)\in\BQ_{N_r}$. Therefore, in light of \eqref{eq:xi-1-general}, $\sqrt{\dim(\CC_r)} = \tau_1(\CC_r)/\xi_1(\CC_r) \in \BQ_{N_{r}}$.
\end{proof}

\begin{rmk}
In the rest of the paper, the $N_r$'s are used to bound the conductors of the cyclotomic fields in which we perform explicit computations. The key information to extract from the above lemma is that $(2r-1)$ is the largest odd factor of $N_r$, while the exact value of $N_r$ is not needed for our purpose.
\end{rmk}

\subsection{Local modules of \texorpdfstring{$\CC_r$}{}} 
In order to study the Witt subgroup generated by $\CC_r$, it is natural to find connected \'etale algebras in it, and to study the Witt class of the corresponding local module category (or condensation), as is explained in Section \ref{subsec:prelim-Witt}. It turns out that $\CC_r$ contains a Tannakian symmetric fusion subcategory, and the corresponding local module admits further decomposition when $r$ is odd.

Details on the above discussions are provided in the following lemma. Recall that a fusion category is unpointed if it has only 1 invertible object, and the semion modular categories is defined in Section \ref{subsec:BFC}.

For any fusion category $\CC$ and object $X \in \CC$, we denote by $\langle X \rangle$ the full fusion subcategory of $\CC$ generated by all the subobjects of $X$. If $W$ is a collection of Witt classes or Witt subgroups, then $\langle W \rangle$ is the subgroup of $\WW$ generated by $W$.

\begin{lem}\label{lem:de-gauging}
For any $r\ge 1$, we have 
\begin{equation*}
\begin{cases}
(\CC_r)_\pt \supset \Rep(\BZ/2\BZ)\,, & \quad \text{if } r \text{ is odd}\,;\\
(\CC_r)_\pt \cong \Rep(\BZ/2\BZ\times\BZ/2\BZ)\,, & \quad \text{if } r \text{ is even}\,.
\end{cases}
\end{equation*}
Let $A_r = \Fun(\BZ/2\BZ)$ for odd $r$ and $A_r = \Fun(\BZ/2\BZ\times\BZ/2\BZ)$ for even $r$ be the corresponding connected \'etale algebras, then we have the following decomposition for the local modules 
\begin{equation*}
(\CC_r)_{A_{r}}^{0}\cong
\begin{cases}
\Sem\bt\DD_{r} \,, & \quad \text{if } r \text{ is odd}\,;\\
\DD_{r}\,, & \quad \text{if } r \text{ is even}\,,
\end{cases}
\end{equation*}
where $\DD_r$ is an unpointed modular category.
\end{lem}
\begin{proof}
Let $\la_1=2r\om_{r-1}$, $\la_2=2r\om_1$, and $\la_3=2r\om_r$. One can check easily that $\la_j \in \fA_r$ for $1\leq j\leq 3$, and by \cite{Saw02,Saw05}, these are all the nontrivial invertible objects in $\CC_r$. Moreover, when $r$ is odd, $\la_1^{\ot j} = \la_j$ for $j = 1, 2, 3$, and when $r$ is even, $\la_i^2=\1$ and $\la_1\otimes \la_2=\la_3$ (see, for example, \cite[Lem.~2]{Saw02}).

Therefore, $(\CC_r)_{\pt} \cong \CC(\BZ/4\BZ,Q_1)$ for odd $r$, $(\CC_r)_{\pt} \cong \CC(\BZ/2\BZ\times \BZ/2\BZ,Q_2)$ for even $r$, where $Q_1$ and $Q_2$ are quadratic forms determined by the following twist values:
\begin{equation}
\begin{split}
&\te_{\la_1}
=q^{(4r\om_{r-1}-2r\om_r|4r\om_{r-1}-2r\om_r+2\rho)}
=q^{2r^2(r-1)+r^2}=\exp\left(\frac{r^2}{2}\pi i\right)\,,\\
&\te_{\la_2}
=q^{(2r\om_1|2r\om_1+2\rho)}
=q^{2r(4r-2)}=\exp\left(2r\pi i\right)=1\,,\\
&\te_{\la_3}
=q^{(2r\om_r|2r\om_r+2\rho)}
=q^{r^2(2r-1)}
=\exp\left(\frac{r^2}{2}\pi i\right)\,.
\end{split}
\end{equation}

When $r$ is odd, the fusion subcategory of $\CC_r$ generated by the object $\la_2$ is equivalent to the symmetric fusion category $\Rep(\BZ/2\BZ)$, which contains the connected \'etale algebra $A_r = \1 \oplus \la_2 = \Fun(\BZ/2\BZ)$. Since $\la_1 \ot \la_2 = \la_3$, so by \eqref{eq:twist-def}, we have 
$$
\theta_{\la_{3}} \cdot \id_{\la_{1}\ot \la_{2}}
=
\theta_{\la_{3}} \cdot \id_{\la_{3}}
= 
\theta_{\la_{1} \ot \la_{2}} 
= 
\beta_{\la_{2},\la_{1}}\circ \beta_{\la_{1}, \la_{2}} \circ (\theta_{\la_{1}}\cdot \id_{\la_{1}} \ot \theta_{\la_2} \cdot \id_{\la_{2}})\,,
$$
where for $j=1, 2, 3$, $\theta_{\la_{j}}$ refers to the twist values above, and $\theta_{\la_{1}\ot\la_{2}}$ denotes the twist isomorphism. Therefore, we have $\beta_{\la_{2},\la_{1}} \circ \beta_{\la_{1},\la_{2}} = \id_{\la_{1}\ot\ld_{2}}$, so $\la_1$ is in the centralizer of $\Rep(\BZ/2\BZ)$. Similarly, we have $\la_3\in\Rep(\BZ/2\BZ)'$.

Let $F:\CC_r \to (\CC_r)_{A_{r}}$, $X \mapsto X\ot A_{r}$ be the free module functor, which is a monoidal functor (see \cite[Sec.~3.3]{DMNO}). Let $V := F(\la_1)$, then $V$ is of order 2. By the above discussions, we have $V \in (\CC_r)_{A_{r}}^0 = (\CC_r)_{\BZ/2\BZ}^{0}$. Moreover, by \cite[Thm.~1.8]{KO02}, we have
$$\dim_{(\CC_{r})_{A_{r}}^{0}}(V) = \dim_{\CC_{r}}(\la_1) = 1\,,$$ 
and by \cite[Thm.~1.17]{KO02}, $\theta_V = \theta_{\la_{1}} = i$. Therefore, $V$ generates a modular subcategory in $(\CC_r)_{A_{r}}^{0}$ which is equivalent to $\Sem$. Thus, by the double centralizer theorem \cite[Thm.~4.2]{Mug03S}, we have $(\CC_r)^0_{A_{r}} \cong \Sem\bt\DD_r$, where $\DD_r$ is the centralizer of $\Sem$ in $(\CC_r)^0_{A_{r}}$. In particular, $\DD_r \cap \Sem = \langle F(\1) \rangle$.

When $r$ is even, $\theta_{\la_{j}} = 1$ for $j = 1, 2, 3$, so \eqref{eq:twist-def} implies that $(\CC_r)_{\pt}$ is equivalent to the symmetric fusion category $\Rep(\BZ/2\BZ\times\BZ/2\BZ)$, which contains the connected \'etale algebra $A_r = \1 \oplus \la_1 \oplus \la_2 \oplus \la_3 = \Fun(\BZ/2\BZ\times\BZ/2\BZ)$. In this case, we can simply set $\DD_r := (\CC_r)_{A_r}^0$. 

It remains to show that $\DD_r$ is unpointed for any $r$. Consider the forgetful functor $G: (\CC_r)_{A_r} \to \CC_r$, which is the adjoint of $F$ (see \cite{KO02}). On the one hand, by \cite[Thm.~1.18]{KO02}, any $X \in \irr((\DD_r)_{\pt})$ must satisfy 
$\dim_{\CC_{r}}(G(X)) = \dim_{\CC_{r}}(A_r) \in \BZ$. On the other hand, by \cite{Sch20}, all the objects in $\CC_r$ with integral dimension are contained in $(\CC_r)_\pt$. Moreover, it is easy to see that any $X \in (\CC_r)_{A_{r}}$ is a subobject of $F(G(X))$. Therefore, for any $X \in \irr((\DD_r)_{\pt})$, we have $X \in F((\CC_r)_{\pt})$. 

By the above discussions, when $r$ is odd, $F((\CC_r)_\pt) = \langle F(\la_1) \rangle \cong \Sem$; when $r$ is even, $F((\CC_r)_\pt) = \langle F(\1) \rangle$ (note that $F(\1)$ is the tensor unit of $(\CC_r)_{A_{r}}^0$). Therefore, $\{F(\1)\} \subset \irr((\DD_r)_\pt) \subset \irr(F((\CC_{r})_\pt)\cap \DD_r) = \{ F(\1)\}$.
\end{proof}

In particular, by \cite{DrGNO}, Lemma \ref{lem:de-gauging} implies that $\FPdim(\CC_r)=8\FPdim(\DD_r)$ if $r$ is odd, and $\FPdim(\CC_r)=16\FPdim(\DD_r)$ if $r$ is even.
Moreover, we immediately see that
\begin{equation}\label{eq:SemD}
[\CC_r] =
\begin{cases}
[\Sem\bt\DD_r] & \text{if }r\text{ is odd}\,, \\
[\DD_r]   & \text{if }r\text{ is even}\,,
\end{cases}
\end{equation}
Now the conformal embedding $\so(m)_n\times \so(n)_m \subset \so(mn)_1$ (see, for example, \cite{DMNO}) implies that 
\[
[\CC_r]^2 =\begin{cases}
[\Sem]^2 & \text{if }r\text{ is odd}\,, \\
[\Vec]   & \text{if }r\text{ is even}\,,
\end{cases}
\]
and so for any $r \ge 1$, we have
\begin{equation}\label{eq:ord2}
[\DD_r]^2 = [\Vec]\,.
\end{equation}

\section{The Witt subgroup generated by \texorpdfstring{$\DD_r$}{}}\label{sec:4}

In this section, we study the Witt subgroup generated by $[\DD_r]$ using the signature homomorphism \eqref{eq:sgn}. By \cite[Lem.~2.4]{Sch18}, $\DD_r$ is pseudounitary, i.e., $\dim(\DD_r) = \FPdim(\DD_r)$, which is a quotient of $\dim(\CC_r)$ by 8 or 16 (see above). For simplicity, we set 
\begin{equation}
    D_r := \sqrt{\dim(\DD_r)} \quad\text{(positive square root)}\,.
\end{equation}
To compute the signature, we will combine the lemmas in Section \ref{sec:lie} with the well-known formulae of $\sqrt{\dim(\CC_r)}$ in terms of trigonometric functions \cite[Thm.~3.3.20]{BK01}. (Note that in \cite[Thm.~3]{Coq10}, a formula for the global dimension is provided, but to compute the signature, one needs to take the square root). As a reminder, the definition of the function $\fd_{r}(j)$ can be found in Lemma \ref{lem:fdrj}.

\begin{lem}\label{lem:Dr-odd}
Let $r \ge 1$ be odd. Then $D_r\in \BQ_{8N_r}$ and
\begin{equation} \label{eq:D}
D_r=Y\sqrt{2r-1}\left(\prod_{j=1}^{2r-3}(\sin\frac{j\pi}{4r-2})^{\mathfrak{d}_r(j)}\right)^{-1}\,,
\end{equation}
where $Y= 2^{\frac{-2r^2+3r+1}{2}} \cdot (2r-1)^{\frac{r-1}{2}} \in \BQ$.
\end{lem}
\begin{proof}
Recall that when $r$ is odd, $ 8\dim(\DD_r) = \dim(\CC_r)$. Therefore, $D_r\in \BQ_{8N_r}$ follows from Lemma \ref{lem:field} and the fact that $\sqrt{2}\in \BQ_{8}$. According to \cite[Thm.~3.3.20]{BK01} and the information on positive roots in Section \ref{sec:lie}, we have
\[
\begin{split}
D_r
&=\frac{1}{2\sqrt{2}}\sqrt{|\fP/((2r+h^{\vee})\fQ^{\vee})|}
\prod_{\al\in \De_+}\left(2\sin(\frac{(\al|\rho)}{2r+h^{\vee}}\pi)\right)^{-1}\\
&=Y\sqrt{2r-1}\prod_{\al\in \De_+}\left(\sin(\frac{(\al|\rho)}{4r-2}\pi)\right)^{-1}\,,
\end{split}
\]
and we are done by definition.
\end{proof}

\begin{lem}
Let $r \ge 2$ be even. Then $D_r\in \BQ_{N_r}$ and
\begin{equation}\label{eq:C}
D_r=Z\left(\prod_{j=1}^{2r-3}(\sin\frac{j\pi}{4r-2})^{\mathfrak{d}_r(j)}\right)^{-1}\,,
\end{equation}
where $Z= 2^{\frac{(3-2r)r}{2}} \cdot (2r-1)^{\frac{r}{2}} \in \BQ$. 
\end{lem}
\begin{proof}
Note that when $r$ is even, $16\dim(\DD_r) = \dim(\CC_r)$. Then the statement of the lemma follows from the similar argument as in Lemma \ref{lem:Dr-odd} and Lemma \ref{lem:field}.
\end{proof}

The above lemmas imply the following periodicity result on the signature, which is important to the Theorems below.

\begin{prop} \label{prop:Period}
Let $r \ge 1$ be any integer. If $k \equiv k' \pmod {4r-2}$ and $(k,4r-2)=1$, then $\sm_k(D_r)=\sm_{k'}(D_r)$, in particular,  $\e_{\DD_r}(\sm_k)=\e_{\DD_r}(\sm_{k'})$.
\end{prop}
\begin{proof}
From Lemma \ref{lem:field}, we have $k,k'$ are both coprime to $8N_r$, so the signatures are well defined. Assume $r$ is odd. When $k\equiv k' \pmod{8r-4}$, from \eqref{eq:Sin} and the fact that $\mid\De_+\mid=r(r-1)$ is even, we have 
$$
\sm_{k'}(D_r)=\left(\jac{-1}{k'}\jac{-1}{k}\right)^{\mid\De_+\mid}\sm_{k}(D_r)=\sm_{k}(D_r).
$$
Therefore, it suffices to consider the case when $k'=4r-2+k$. Without loss of generality, assume $k\equiv 1\pmod{4},\ k'\equiv 3\pmod{4}$. By assumption, $r$ is odd, so $2r-1\equiv 1 \pmod 4$, and we have $\sm_{k}(\sqrt{2r-1})=\sm_{k'}(\sqrt{2r-1})$ since $\sqrt{2r-1}\in \rational_{2r-1}$ (\cite{Was82}). By Equation \ref{eq:D}, Lemmas \ref{Lem:odd in m} and \ref{lem:sin}, we have
\[
\begin{split}
\sm_{k'}(D_r)&=\jac{-1}{k'}^{\mid\De_+\mid}Y\sqrt{2r-1}\left(\prod_{j=1}^{2r-3}\left(\sin\frac{(k+4r-2)j\pi}{4r-2}\right)^{\mathfrak{d}_r(j)}\right)^{-1}\\
&=Y\sqrt{2r-1}\left(\prod_{j=1}^{2r-3}\left((-1)^j\sin\frac{kj\pi}{4r-2}\right)^{\mathfrak{d}_r(j)}\right)^{-1}\\
&=(-1)^{\mid S_r\mid}\sm_{k}(D_r)\\
&=\sm_{k}(D_r)\,.\qedhere
\end{split}
\]
When $r$ is even, the proof is similar.
\end{proof}

In the following, we determine the structure of the Witt subgroup generated by families of $[\DD_r]$'s. The key point is to show that there are no nontrivial relations in the groups under study. To achieve this goal, we use special Galois elements to distinguish the signature homomorphisms of the $[\DD_r]$'s, and the parity of $r$ determines the way to find the elements. Naturally, from now on, we separate cases according to the parity of $r$.

\subsection{The case when \texorpdfstring{$r$}{} is odd}
In this subsection, let $r \ge 1$ be an odd integer.

\begin{prop} \label{prop:Dr-sign}
For any positive integer $r \equiv 5 \pmod 8$, we have
$
\e_{\DD_r}(\sm_r)=-1\,. 
$
\end{prop}

\begin{proof}
From Lemma \ref{lem:field}, $(r,8N_r)=(r,2r-1)=1$. Let $r=4k+1$ for some odd integer $k$. By \eqref{eq:m}, it is clear that $\mathfrak{d}_r(j)$ is odd when $j\in \{4m+1,4(m+1),4n+2,4n+3\ |\ 0\leq m\leq k-1,\  k\leq n\leq 2k-1\}$. By Lemma \ref{lem:Dr-odd}, \cite[Lem.~6.2, 6.3]{NRWZ20a} and the quadratic reciprocity, we have: 
\[
\begin{split}
\e_{\DD_r}(\sm_r)
&=\jac{r}{2r-1}(-1)^{\sum\limits^{k-1}_{m=0}(\lfloor\frac{(4m+1)r}{4r-2}\rfloor+\lfloor\frac{4(m+1)r}{4r-2}\rfloor)+\sum\limits^{2k-1}_{n=k}(\lfloor\frac{4n+2}{4r-2}\rfloor+\lfloor\frac{(4n+3)r}{4r-2}\rfloor)}\\
&=\jac{-1}{r}(-1)^{\sum\limits_{m=0}^{k-1}(m+m+1)+\sum\limits_{n=k}^{2k-1}(n+n)}
=(-1)^k =-1\,. \qedhere
\end{split}
\]
\end{proof}

Now we are ready to determine the subgroup of $\WW$ generated by $\DD_r$ for a collection of odd $r$'s. Let $\{a_j\}_{j\geq 1}$ be a sequence of prime numbers such that $a_j \equiv 9 \pmod {16}$ and $a_j<a_{j+1}$. Such a sequence exists due to the Dirichlet prime number theorem. Let $r_j=\frac{a_j+1}{2}$, then $r_j \equiv 5 \pmod 8$ and $a_j=2r_j-1$. 

\begin{thm} \label{mutural distinctness}
The Witt subgroup generated by $\{[\DD_{r_j}]\}_{j\geq 1}$ is isomorphic to $\BZ/2\BZ^{\op\BN}$.
\end{thm}

\begin{proof}
Since $\DD_r$ has order 2 in the Witt group (see \eqref{eq:ord2}), it suffices to show that no finite product of distinct $\DD_{r_j}$ yields $[\Vec]$. To that end, let $\DD:=\bt_{i=1}^{n}\DD_{r_{j_i}}$ for some subsequence $1 \le j_1 < \cdots < j_n$. From Lemma \ref{lem:field} and Lemma \ref{lem:Dr-odd}, we have $\sqrt{\dim(\DD)}\in \BQ_{N}$, where $N=\lcm_{1\leq i\leq n}\{8N_{r_{j_i}}\}$. Suppose $[\DD] = [\Vec]$, then by \cite[Thm.~3.5]{NRWZ20a}, for any $\sm\in \gal(\overline{\BQ}/\BQ)$, we have 
\[\e_{\DD}(\sm)=1\,.\]
Let $K=\lcm_{2\leq i\leq n}\{2a_{j_i}\}$, then $(a_{j_1},K)=1$ by construction. Hence, by Bezout's identity, there exist $x,y\in \BZ$ such that
$$xa_{j_1}+yK=1\,.$$
Let $k=-(r_{j_1}-1)xa_{j_1}+r_{j_1}$. Since $2\mid (r_{j_1}-1)$, we have
\begin{equation}\label{eq:k_value}
 k\equiv 1 \pmod{K}\,,\quad k\equiv r_{j_1}\pmod{2a_{j_1}}\,. 
\end{equation}
In particular, by Lemma \ref{lem:field} $\gcd(k, N) = 1$. 

Since $2a_{j_i}=4r_{j_i}-2$, by Proposition \ref{prop:Period} and Proposition \ref{prop:Dr-sign}, there exists $\sm\in \GQ$, such that $\sm|_{\BQ_{N}}=\sm_{k}$, and
$$
\e_{\DD_{r_{j_1}}}(\sm)=-1 \ \text{and}\ \e_{\DD_{r_{j_i}}}(\sm)=1\  (i\neq 1)\,.
$$
This implies that
$
\e_{\DD}(\sm)=-1,
$
which is a contradiction.
\end{proof}

Consider $\II := (\sl_2)_2$. It is an Ising modular category, i.e., $\II$ is not pointed and $\FPdim(\II) = 4$, and is extensively studied in \cite[Appendix B]{DrGNO}. Since $[\II]^2 \in \WW_\pt$ and both $\WW_{\pt}$ and $\langle[\II]\rangle$ are contained in the 2-torsion part of $\WW$ (see \cite{DMNO}), it is natural to compare these Witt subgroups with the one we have studied. Let $H$ denote the group generated by $\{[\DD_{r_j}]\}_{j\geq 1}$.

\begin{thm}\label{Not inside group odd}
We have $H\cap \langle\WW_\pt,[\II]\rangle= \{[\Vec]\}$.
\end{thm}

\begin{proof}
It suffices to show product of any finite set in $\{[\DD_{r_j}]\}$ intersect trivially with $\langle\WW_\pt,[\II]\rangle$. To that end, let $\DD:=\bt_{i=1}^{n}\DD_{r_{j_i}}$ for some $1 \le j_1 < \cdots < j_n$. If $[\DD] \in H\cap \WW_\pt$ is nontrivial, there exists a finite abelian group $L$ and a non-degenerate quadratic form $Q: L\to \BC^{\times}$ such that the corresponding pointed modular category $\LL=\CC(L,Q)$ satisfies $[\LL]=[\DD]$ in $\WW$. Consequently, by \cite[Thm.~3.5]{NRWZ20a}, for any $\sm\in \gal(\overline{\BQ}/\BQ)$, we have
$$
\e_{\DD}(\sm)=\e_{\LL}(\sm)\,.
$$
Let $h=|L|$, $N=h\cdot \lcm_{1\leq i\leq n}\{8N_{r_{j_i}}\}$, then both $\sqrt{h}$ and $\sqrt{\dim(\DD)}$ are in the field $\BQ_{N}$. We write $h=h_1h_2$, where $(h_1, a_{j_1})=1$, $h_2=a_{j_1}^s$ for some $s \ge 1$. Let $K=h_1 \cdot \lcm_{2\leq i\leq n}\{2a_{j_i}\}$, we have $(a_{j_1}, K)=1$. Hence, from Bezout’s identity, there exist $x,y\in \BZ$, such that:
$$
xa_{j_1}+yK=1
$$
Let $k=-(r_{j_1}-1)xa_{j_1}+r_{j_1}$, since $2|(r_{j_1}-1)$, we have:
$$
k\equiv 1 \pmod{K}\,,\quad k\equiv r_{j_1}\pmod{2a_{j_1}}\,.
$$
Again by Lemma \ref{lem:field}, $\gcd(k,N)=1$. By Proposition \ref{prop:Period} and Proposition \ref{prop:Dr-sign}, there exists $\sm\in \GQ$, such that $\sm|_{\BQ_{N}}=\sm_{k}$ and

$$
\e_{\DD_{r_{j_1}}}(\sm)=-1,\ \text{and}\  \e_{\DD_{r_{j_i}}}(\sm)=1\  (i\neq 1)\,,
$$
which implies
$\e_{\DD}(\sm)=-1$.

Since $k \equiv 1 \pmod{4}$, $k\equiv 1 \pmod{h_1}$ and $r_{j_1} \equiv 1 \pmod 4$, so by \cite[Lem.~6.2]{NRWZ20a}, we have
$$
\begin{aligned}
\e_{\LL}(\sm)
&= \sgn(\sm_k(\sqrt{h_1h_2}))
= \jac{k}{h_1}\jac{k}{h_2}
= \jac{k}{a_{j_1}}^s\\
&= \jac{r_{j_1}}{2r_{j_1}-1}^s
=\jac{2r_{j_1}-1}{r_{j_1}}^s
=\jac{-1}{r_{j_1}}^s
=1\,,
\end{aligned}
$$
a contradiction. So $H\cap \WW_\pt$ is trivial.

Since $[\II]^2\in \WW_\pt$, if $[\DD] = [\II]^m[\HH']$ for some $[\HH']\in \WW_\pt$, then by the above discussions, we must have $m$ is odd. Let $\kappa = \xi_1(\II)$, then $\kappa$ is a primitive 16-th root of unity. Moreover, it is well-known that the first central charge of a pointed modular category is an 8-th root of unity (see \cite{DLN15}). Consequently, there exists some $\ell \in \BZ$ such that $\xi_1(H') = \kappa^\ell$. Moreover, by \eqref{eq:SemD} and the multiplicativity of the central charge (\cite{DrGNO}), we have $ \xi_1(\DD_r)=1$. Therefore,
$$
1 = \xi_1(\DD) = \xi_1(\II)^m \xi_1(\HH') = \kappa^m \cdot \kappa^{2\ell} = \kappa^{m + 2\ell}\,.
$$
However, as $m$ is odd, the right hand side can never be 1, so we get a contradiction. Thus, we must have $H\cap \langle\WW_\pt,[\II]\rangle= \{[\Vec]\}$.
\end{proof}

In view of \eqref{eq:SemD}, \eqref{eq:ord2}, and combining all the results above, we have the following corollary.
\begin{cor}
The image of $[\Vec]$ has infinitely many square roots in $\WW/\langle\WW_\pt,[\II]\rangle$, hence the same is true of $[\Sem]^2$.\qed
\end{cor}

\subsection{The case when \texorpdfstring{$r$}{} is even}
In this subsection, let $r \ge 2$ be an even integer. 

\begin{prop}\label{prop:Cr-sign}
Let $r \equiv 4 \pmod 8$, $s\equiv 6 \pmod 8$ we have
\[ 
\e_{\DD_r}(\sm_{2r+1})=\e_{\DD_{s}}(\sm_{s-1})=-1\,. 
\]
\end{prop}

\begin{proof}
From Lemma \ref{lem:field}, $\gcd(2r+1,N_r)=\gcd(2r+1,2r-1)=1$ and $\gcd(s-1,N_{s})=\gcd(s-1,2s-1)=1$. Similar to the proof of Proposition \ref{prop:Dr-sign},
let $r=4k$, $s=4l+2$ for some odd integers $k$, $l$. By \eqref{eq:m}, it is clear that $\mathfrak{d}_r(j)$ is odd when $j\in \{4m+2,4m+3,4n,4n+1\ |\ 0\leq m\leq k-1,\  k\leq n\leq 2k-1\}$ and $\mathfrak{d}_{s}(j)$ is odd when $j\in \{4m+2,4m+3,4n,4n+1\ |\ 0\leq m\leq l-1,\  l+1\leq n\leq 2l\}$. Since $2r+1 \equiv 1\pmod 4,\ l-1\equiv 1 \pmod 4$, by \eqref{eq:C}, \cite[Lem.~6.1]{NRWZ20a}, we have: 
\[\begin{split}
\e_{\DD_r}(\sm_{2r+1})
&=(-1)^{\sum\limits^{k-1}_{m=0}(\lfloor\frac{(4m+2)(2r+1)}{4r-2}\rfloor+\lfloor\frac{4m+3)(2r+1)}{4r-2}\rfloor)+\sum\limits^{2k-1}_{n=k}(\lfloor\frac{4n(2r+1)}{4r-2}\rfloor+\lfloor\frac{(4n+1)(2r+1)}{4r-2}\rfloor)}\\
&=(-1)^{\sum\limits_{m=0}^{k-1}(2m+1+2m+1)+\sum\limits_{n=k}^{2k-1}(2n+1+2n)}\\
&=(-1)^k=-1\
\end{split}\]
and
\[\begin{split}
\e_{\DD_{s}}(\sm_{s-1})
&=(-1)^{\sum\limits^{l-1}_{m=0}(\lfloor\frac{(4m+2)(s-1)}{4r-2}\rfloor+\lfloor\frac{4m+3)(s-1)}{4r-2}\rfloor)+\sum\limits^{2l}_{n=l+1}(\lfloor\frac{4n(s-1)}{4r-2}\rfloor+\lfloor\frac{(4n+1)(s-1)}{4r-2}\rfloor)}\\
&=(-1)^{\sum\limits_{m=0}^{l-1}(m+m)+\sum\limits_{n=l+1}^{2l}(n-1+n)}\\
&=(-1)^l=-1\,.\qedhere
\end{split}\]
\end{proof}

Let $\{b_j\}_{j\geq 1}$, $\{c_j\}_{j\geq 1}$ be two increasing sequences of prime numbers such that $b_j \equiv 7 \pmod {16}$ and $c_j \equiv 11 \pmod {16}$. Such sequences exist due to the Dirichlet prime number theorem. Let $r_j=\frac{b_j+1}{2}$, $s_j=\frac{c_j+1}{2}$, then $r_j \equiv 4 \pmod 8$ and $s_j \equiv 6 \pmod 8$.

\begin{thm} \label{C_r distinct}
$\{[\DD_{r_j}]\}_{j\geq 1}$ and $\{[\DD_{s_j}]\}_{j\geq 1}$ each generate a group isomorphic to $\BZ/2\BZ^{\op\BN}$ in the Witt group, and the intersection of these groups is trivial.
\end{thm}

\begin{proof}

Similar to the argument as in Theorem \ref{mutural distinctness}, it suffices to show that no finite product of distinct $\DD_{r_j}$ or $\DD_{s_j}$ yields $[\Vec]$. To that end, let $\DD:=\bt_{i=1}^{n}\DD_{r_{j_i}}$, and $\Tilde{\DD}:=\bt_{i=1}^{n}\DD_{s_{l_i}}$ for arbitrary finite subsets of $\{r_j\}$ and $\{s_j\}$, respectively.  We have $\sqrt{\dim(\DD)}\in \BQ_N$, $\sqrt{\dim(\tilde{\DD})}\in\BQ_{\tilde{N}}$, where $N=\lcm_{1\leq i\leq n }\{N_{r_{j_i}}\}$ and $\tilde{N}=\lcm_{1\leq i\leq n }\{\tilde{N}_{s_{l_i}}\}$. Suppose $[\DD]=[\Tilde{\DD}] = [\Vec]$, then by \cite[Thm.~3.5]{NRWZ20a}, for any $\sm\in \gal(\overline{\BQ}/\BQ)$, we have 
$$\e_{\DD}(\sm)=\e_{\Tilde{\DD}}(\sm)=1\,.$$ 
Let $K=\lcm_{2\leq i\leq n}\{2b_{j_i}\},\  M=\lcm_{2\leq i\leq n}\{2c_{l_i}\}$, then $(b_{j_1},K)=1,\ (c_{l_1},M)=1$ by construction. Hence, by Bezout's identity, there exist $x,y,z,w\in \BZ$ such that
$$xb_{j_1}+yK=1\quad\quad zc_{l_1}+wM=1\,.$$
Let $k=-(2r_{j_1})xb_{j_1}+2r_{j_1}+1$, $t=-(s_{l_1}-2)zc_{l_1}+s_{l_1}-1$. Since $2\mid 2r_{j_1}$ and $2\mid s_{l_1}-2$, we have
$$
\begin{aligned}
&k\equiv 1 \pmod{K}\,,\quad k\equiv r_{j_1}\pmod{2b_{j_1}}\,,\\
&t\equiv 1 \pmod{M}\,,\quad t\equiv s_{l_1}-1\pmod{2c_{l_1}}\,.
\end{aligned}
$$
Since $2b_{j_i}=4r_{j_i}-2,\ 2c_{l_i}=4s_{l_i}-2$, we have $(k,N)=(t,\tilde{N})=1$ (Lemma \ref{lem:field}). By Proposition \ref{prop:Period} and Proposition \ref{prop:Dr-sign}, there exist $\sm$, $\eta\in \gal(\ol{\BQ}/\BQ)$ such that $\sm|_{\BQ_{N}}=\sm_{k}$, $\eta|_{\BQ_{\tilde{N}}}=\eta_{t}$
$$
\e_{\DD_{r_{j_1}}}(\sm)=\e_{\DD_{s_{l_1}}}(\eta)=-1 \ \text{and}\ \e_{\DD_{r_{j_i}}}(\sm)=\e_{\DD_{s_{l_i}}}(\eta)=1\  (i\neq 1)\,.
$$
This implies that
$$
\e_{\DD}(\sm)=\e_{\Tilde{\DD}}(\eta)=-1,
$$
which is a contradiction.\\
Moreover, by considering the products of the categories in two different families, and the fact that $b_j,c_{j}$ are distinct primes, the trivial intersection follows from the same argument as above (consider the least common multiple for all $2b_{j_i}$s and $2c_{l_i}$s except $2b_{j_1}$). 
\end{proof}

Let $H'$ denote the group generated by $\{[\DD_{r_j}]\}_{j\geq 1}$ and $ \{[\DD_{s_j}]\}_{j\geq 1}$. Similar to Theorem \ref{Not inside group odd}, we have

\begin{thm}\label{Not inside group even}
We have $H'\cap \langle\WW_\pt,[\II]\rangle= \{[\Vec]\}$.
\end{thm}
\begin{proof}
From the proof of Theorem \ref{Not inside group odd}, \ref{C_r distinct} and Proposition \ref{prop:Cr-sign}, it suffices to check for any $j$, $2r_j+1\equiv 1 \pmod 4$, $s_j-1\equiv 1 \pmod 4$ and $\jac{2r_j+1}{2r_{j}-1}=\jac{s_j-1}{2s_{j}-1}=1$. Indeed,
$\jac{2r_j+1}{2r_{j}-1}=\jac{2}{2r_{j}-1}=1$ (since when $p\equiv \pm1 \pmod 8$,\  $\jac{2}{p}=1$), and others are also easy to check.

\end{proof}

\section{Relation with Witt subgroup generated by \texorpdfstring{$\so(2r+1)_{2r+1}$}{}}
Let $\BB_b=\so(2b+1)_{2b+1}$, where $b>1$ is an odd integer. In \cite{NRWZ20a}, the Witt subgroup generated by an infinite sequence of $\BB_b$'s, denoted by $G_{\mathbf{p}}$, was studied. Here, $\mathbf{p}$ is a sequence of odd primes satisfying certain congruence conditions in \cite[Section ~6.2]{NRWZ20a}. It is shown in \cite[Theorem~6.7]{NRWZ20a} that $G_{\mathbf{p}}$ contains an elementary 2-group of infinite rank, so it is natural to compare it with the Witt subgroups in this paper.  

In this section, we prove there is an infinite subgroup of the group generated by $[\DD_{r}]$ with $r\equiv 4\pmod{8}$ that intersects trivially with the group $G_{\mathbf{p}}$. The tool we use is still the signature homomorphism. Thus, we start with the square root of the global dimension of $\BB_b$ (cf.~\cite[Prop.~5.3]{NRWZ20a})

\begin{equation}\label{eq:CC}
B_b := \sqrt{\dim(\BB_b)}=
W\sqrt{b}\left(\prod_{l=1}^{b}(\sin\frac{(2l-1)\pi}{8b})\prod_{j=1}^{2b-2}(\sin\frac{j\pi}{4b})^{\mathfrak{c}_b(j)}\right)^{-1}\,,  
\end{equation}
where $W=\frac{b^{\frac{b-1}{2}}}{2^{b^2-b-1}}\in \BQ$ and $\mathfrak{c}_b(j)=b-\lceil\frac{j}{2}\rceil$ for $1\leq j\leq 2b-2$. 

\begin{lem} \label{lem:Bfield}
We have 
\begin{itemize}
\item[(i)] 
$B_b\in \BQ_{16b}$.
\item[(ii)]
Let $k$ be any integer such that $k\equiv 1\pmod{4}$ and $\gcd(k,b)=1$. For any $x \in \BZ$, if we denote $k' = 8xb+k$, then $\e_{\BB_b}(\sm_{k'})=(-1)^x\e_{\BB_b}(\sm_{k})$.
\end{itemize}
\end{lem}

\begin{proof}
Since $\sqrt{b}\in \BQ_{4b}$, $\sin(\frac{j\pi}{4b})\in \BQ_{8b}$ and $\sin(\frac{(2l-1)\pi}{8b})\in \BQ_{16b}$, so $B_b\in \BQ_{16b}$. 

By assumption, $k\equiv 1\pmod 4$ and $k,k'$ are both coprime to $16b$, so the signatures are well defined. Moreover, by the similar argument as Proposition \ref{prop:Period}, together with \cite[Lemma~6.1]{NRWZ20a}, we immediately have (note that $b$ is odd)
\[
\e_{\BB_b}(\sm_{k'})=(-1)^{bx}\e_{\BB_b}(\sm_{k})=(-1)^{x}\e_{\BB_b}(\sm_{k})\,.\qedhere
\]
\end{proof}
\begin{prop}\label{compare signature}
Let $r\equiv 12\pmod{80}$, $b=2r-1\equiv 23 \pmod{160} $ and $x\in \BZ$. We have
\[ 
\e_{\DD_{r}}(\sm_{r-3})=1,\quad\e_{\BB_{b}}(\sm_{8xb+r-3})=(-1)^{x}\,. 
\]
\end{prop}

\begin{proof}
Suppose $r=80k+12$ for some integers $k\geq 0$ and $b=160k+23$. Since 
$r-3\equiv 1 \pmod 8$, $r-3$ is coprime to both $8N_r$ and $16b$. We have seen in Proposition \ref{prop:Cr-sign} that $\mathfrak{d}_{r}(j)$ is odd when $j\in \{4m+2,4m+3,4n,4n+1\ |\ 0\leq m\leq 20k+2,\  20k+3\leq n\leq 40k+5\}$, and it is straight forward to check $\mathfrak{c}_{b}(j)$ is odd when $j\in\{4m+3,4(m+1)\mid 0\leq m\leq 80k+10\} $. By \eqref{eq:C}, \eqref{eq:CC}, \cite[Lem.~6.2, 6.3]{NRWZ20a} (as $r-3\equiv 1 \pmod 8$) and the quadratic reciprocity, we have 
\[
\e_{\DD_{r}}(\sm_{r-3})
=(-1)^{\sum\limits^{20k+2}_{m=0}(\lfloor\frac{(4m+2)(r-3)}{4r-2}\rfloor+\lfloor\frac{4m+3)(r-3)}{4r-2}\rfloor)+\sum\limits^{40k+5}_{n=20k+4}(\lfloor\frac{4n(r-3)}{4r-2}\rfloor+\lfloor\frac{(4n+1)(r-3)}{4r-2}\rfloor)}
\]
and
\[\begin{split}
\e_{\BB_{b}}(\sm_{r-3})
&=\jac{r-3}{b}(-1)^{\sum\limits^{b}_{l=1}(\lfloor\frac{(2l-1)(r-3)}{8b}\rfloor+\sum\limits^{80k+10}_{m=0}(\lfloor\frac{(4m+3)(r-3)}{4b}\rfloor+\lfloor\frac{(4(m+1)(r-3)}{4b}\rfloor)}\,.
\end{split}\]
Now we compute the exponents in the above equations term by term.
\[\begin{split}
&\lfloor\frac{(4m+2)(r-3)}{4r-2}\rfloor=m+\lfloor\frac{2r-10m-6}{4r-2}\rfloor
=
\begin{cases}
m & \text{ if } 0\leq m\leq \lfloor\frac{2r-6}{10}\rfloor=16k+1\,,\\
m-1 & \text{ if } 16k+2\leq m \leq 20k+2\,,
\end{cases}\\
&\lfloor\frac{(4m+3)(r-3)}{4r-2}\rfloor=m+\lfloor\frac{3r-10m-9}{4r-2}\rfloor=m\,,\\
&\lfloor\frac{(4n)(r-3)}{4r-2}\rfloor=n-1+\lfloor\frac{10n-4r+2}{4r-2}\rfloor
=
\begin{cases}
n-2 & \text{if } 20k+3\leq m\leq 32k+4\,,\\
n-1 & \text{if } 32k+5\leq m \leq 40k+5\,,
\end{cases}\\
&\lfloor\frac{(4n+1)(r-3)}{4r-2}\rfloor=n-1+\lfloor\frac{10n-5r+5}{4r-2}\rfloor
=n-2\,.
\end{split}\]
Therefore, we have 
\[\e_{\DD_{r}}(\sm_{r-3})
=(-1)^{\sum\limits^{20k+2}_{m=16k+2}(m-1+m)+\sum\limits^{40k+5}_{n=32k+5}(n-1+n-2)}
=(-1)^{(4k+1)+(2k+1)}
=1\,.\]

To facilitate the various cases, write $m=2m'+s$ for $s\in \{0,1\}$, and write $l=8l'+t$ for $t\in \{1,2,3,4,5,6,7,8\}$, then we have (i) $0\leq m'\leq 40k+5$ when $s=0$; (ii) $0\leq m'\leq 40k+4$ when $s=1$; (iii) $0\leq l'\leq 20k+2$ when $t\in \{1,2,3,4,5,6,7\}$ and (iv) $0\leq l'\leq 20k+1$ when $t=8$. Therefore,

\[\begin{split}
\lfloor\frac{(4m+3)(r-3)}{4b}\rfloor=&m'+\lfloor\frac{(240+320s)k-20m+27+36s}{640k+92}\rfloor\\
=&
\begin{cases}
m' & s \in \{0, 1\} \text{ and } \  0\leq m'\leq (12+16s)k + 2s+1\,,\\
m'-1 & s \in \{0, 1\}\text{ and } (12+16s)k+2(s+1)\leq m' \leq 40k+5-s\,.
\end{cases}\\
\end{split}\]

\[\begin{split}
\lfloor\frac{4(m+1)(r-3)}{4b}\rfloor=&m'+\lfloor\frac{(320+320s)k-20m+36+36s}{640k+92}\rfloor\\
=&
\begin{cases}
m' & s \in \{0, 1\} \text{ and }  0\leq m'\leq 16(1+s)k+2s+1\,,\\
m'-1 & s \in \{0, 1\} \text{ and } 16(1+s)k+2(1+s)\leq m' \leq 40k+5-s\,.
\end{cases}\\
\end{split}\]

\[\begin{split}
\lfloor\frac{(2l-1)(r-3)}{8b}\rfloor=&l'+\lfloor\frac{(160t-80)k+18t-9-40l'}{1280k+184}\rfloor\\
=&
\begin{cases}
l' & 1 \le t \le 5 \text{ and }  0\leq l'\leq (4t-2)k + \lfloor\frac{t-1}{2}\rfloor\,,\\
l'-1 & 1\le t \le 5 \text{ and }  (2+4t)k + \lfloor\frac{t+1}{2}\rfloor\leq l' \leq 20k+2\,,\\
l' & t\in \{6, 7\}  \text{ and } 0\leq l'\leq 20k+2\,,\\
l' & t=8 \text{ and } 0\leq l'\leq 20k+1\,.
\end{cases}\\
\end{split}\]

Therefore, we have
\[\begin{split}
\e_{\BB_{b}}(\sm_{r-3})
=&\jac{80k+9}{160k+23}(-1)^{\sum\limits^{28k+3}_{m'=12k+2}(m'-1+m')+40k+5-1+\sum\limits^{32k+3}_{m'=16k+2}(m'-1+m')+40k+5-1}\\
\times&(-1)^{\sum\limits^{2k+1}_{m'=6k}(l'-1+l')+\sum\limits^{14k+1}_{m'=10k+2}(l'-1+l')+\sum\limits^{20k+2}_{m'=18k+3}(l'-1+l')+20k+2}\\
=&\jac{5}{80k+9}(-1)^{(16k+2)+40k+4+(16k+2)+40k+4+(4k)+(4k)+(2k)+20k+2}\\
=&\jac{9}{5}
=1\,.
\end{split}\]

From Lemma \ref{lem:Bfield} ($r-3\equiv 1 \pmod 8$), we have 
\[
\e_{\BB_{b}}(\sm_{8x+(r-3)})=(-1)^{x}\e_{\BB_{b}}(\sm_{r-3})=(-1)^x\,.\qedhere
\]
\end{proof}
Let $\{b_i\}_{i\geq 1}$ be a sequence of prime numbers such that $b_i\equiv23 \pmod{160}$, which exists due to the Dirichlet prime number theorem. Let $r_i=\frac{b_i+1}{2}\equiv 12 \pmod{80}$, and we use $H_{\mathbf{r}}$ to denote the group generated by $\{[\DD_{r_i}]\}_{i\geq 1}$. Let $G_{\mathbf{p}}$ be the Witt subgroup at the beginning of this section. Note that the set $\{b_i\}$ may intersect with $\mathbf{p}$.
\begin{thm} 
With the above notations, we have $G_{\mathbf{p}}\cap H_{\mathbf{r}}=\{[\Vec]\}$.
\end{thm}

\begin{proof}
It suffices to show that no finite product of
distinct $\DD_{r_j}$'s and elements in $G_{\mathbf{p}}$ yields $[\Vec]$ (elements in $G_{\mathbf{p}}$ may have multiplicities in this product). To this end, we take an arbitrary finite subsequence of $\{r_i\}$ denoted by $\{r_{j_{w}}\mid 1 \le w \le m\}$. We also pick an arbitrary finite subsequence $\{p_l \mid 1 \le l \le n\}$ of $\mathbf{p}$ such that $2r_{j_{1}} - 1 \ne p_l$ for any $1 \le l \le n$. In other words, $p_l \ne b_{j_{1}}$ for any $1 \le l \le n$.

Now for any $s, \gamma_l \in \BZ$, we consider the finite product \[\FF:=\left(\bt_{w=1}^{m}{\DD_{r_{j_w}}}\right)\bt\left(\bt_{l=1}^{n}\BB^{\bt\gamma_l}_{p_{l}}\right)\bt \left(\BB_{b_{j_{1}}}^{\bt s}\right)\,.\]

By the previous results, we have $\sqrt{\dim(\DD)}\in \BQ_N$, where $N=\lcm_{1\leq w\leq m\,,~1\leq l\leq n}\{32N_{r_{j_w}},32p_{l}\}$. Suppose that $[\EE]=[\Vec]$, then by \cite[Thm.~3.5]{NRWZ20a}, for any $\mu\in \gal(\overline{\BQ}/\BQ)$, then we would have 
\[\e_{\FF}(\mu)=1\,.\]

Let $K=\lcm_{2\leq w\leq m\,,~1\leq l\leq n}\{32b_{j_w},32p_{l}\}$, then $(b_{j_1},K)=1$ by construction. Hence, by Bezout's identity, there exist $x,y\in \BZ$ such that
\begin{equation}\label{eq:Bezout}
 xb_{j_1}+yK=1\,.
\end{equation}

Define
\[t=\begin{cases}
-(r_{j_1}-4)xb_{j_1}+r_{j_1}-3 & \text{ if } s \text{ is odd}\,,\\
-(2r_{j_1})xb_{j_1}+2r_{j_1}+1 & \text{ if } s \text{ is even}\,.
\end{cases}\]
Since $r_{j_1}-4=8(10k+1)$ and $4\mid 2r_{j_1}$, we have
\[t\equiv 1 \pmod{K}
\quad\text{and}\quad 
t\equiv \begin{cases}
r_{j_1}-3\pmod{2b_{j_1}} & \text{ if } s \text{ is odd}\,,\\
2r_{j_1}+1\pmod{2b_{j_1}} & \text{ if } s \text{ is even}\,.
\end{cases}\]
Again by Lemma \ref{lem:field}, we have $(t,N)=1$. Moreover, we use Propositions \ref{prop:Period}, \ref{prop:Cr-sign} and \ref{compare signature} and get the following two cases (note that $x$ is odd by construction). To avoid too many subscripts in the expression of the signatures, we temporarily write $\e^{\DD}_{w} := \e_{\DD_{r_{j_{w}}}}$ for $1 \le w \le m$, $\e^{\BB}_{l}:=\e_{\BB_{p_{l}}}$ for $1 \le l \le n$, and we write $\e^{\BB}_{j_{1}} := \e_{\BB_{j_{1}}}$ in this proof. Note that $\e_{\BB^{\bt s}_{b_{j_{1}}}}=(\e^{\BB}_{j_{1}})^s$.

Firstly, when $s$ is odd, there exists $\sm\in \GQ$, such that $\sm|_{\BQ_{N}}=\sm_{t}$, and
\[
\begin{split}
&\e^{\DD}_{w}(\sm) =  \e^{\BB}_{l}(\sm)=1
\quad\text{ for } 1\le w \le m\,,\ 1 \le l \le n\,,\\
&\e^{\BB}_{j_{1}}(\sm_{8(10k+1)xb_{j_1}+r_{j_1}-3})^s=(-1)^{(10k+1)xs}=-1\,.
\end{split}
\]
Secondly, when $s$ is even, there exists $\eta\in \GQ$, such that $\eta|_{\BQ_{N}}=\eta_{t}$, and
\[
\e^{\DD}_{w}(\eta) = \e^{\BB}_{l}(\eta)=\e^{\BB}_{j_{1}}(\eta)^s=1 \text{ for } 2\le w \le m\,,\ 1 \le l \le n\,,\text{ and } \e^{\DD}_{1}(\eta)=-1\,.
\]
Therefore, in either case, $\e_\FF = \prod_{w=1}^{m} \e^\DD_w \cdot \prod_{l=1}^{n} (\e^\BB_l)^{\gamma_{l}} \cdot (\e^{\BB}_{j_{1}})^s$ is not the constant function 1, which is a contradiction.
\end{proof}

\section{An order two Witt class with property \texorpdfstring{$S$}{}}
Recall that a non-degenerate braided fusion category $\CC$ has \emph{property S} if it is completely anisotropic, simple and unpointed. The subgroup of $\WW$ generated by the categories with property $S$ is denoted by $\WW_{S}$.
In this section, we prove that $[\DD_4]\in \WW_{S}$. Since $\DD_4$ has order 2 in $\WW$ (see \eqref{eq:ord2}), it gives a positive answer to the open question \cite[Question~6.8]{DMNO}. From Theorem \ref{Not inside group even}, $[\DD_4]\notin \WW_\pt$, hence it suffices to prove $\DD_4$ is simple and completely anisotropic. 

\begin{lem}\label{Lem:Simplicity}
For any positive integer $r\geq 2$, the category $\DD_r$ is simple. 
\end{lem}

\begin{proof}
From Lemma \ref{lem:de-gauging} and \cite[Lemma~1]{Sch20}, we have $(\DD_r)_{\pt}\cong\Vec$, and there is no nontrivial objects in $\DD_r$ of integer dimension. Hence if $\RR\subset \DD_r$ is a fusion subcategory, then by Deligne's Theorem \cite[Section~9.9]{EGNO} and above discussion, we have $\RR'\subset (\DD_r)_{\pt}\cong\Vec$, which implies $\RR$ is non-degenerate. Now the Lemma follows from \cite[Theorem~2]{Sch20}.
\end{proof}

Now let $r = 4$. Recall from Lemma \ref{lem:de-gauging} that in this case, $(\CC_4)_\pt \cong \Rep(\BZ/2\BZ\times\BZ/2\BZ)$, which contains the connected \'etale algebra $A_4 =
\Fun(\BZ/2\BZ\times\BZ/2\BZ)$. By definition, $\DD_4 = (\CC_4)^{0}_{A_{4}}$. We will show that $\DD_4$ is completely anisotropic, i.e., it does not have any nontrivial connected \'etale algebra.

By \cite[Lem.~2.24]{Sch18}, all connected \'etale algebras in a pseudounitary braided fusion category have trivial twists. Therefore, it suffices to consider only objects with trivial twists in $\DD_4$. By \cite[Thm.~1.17]{KO02}, such objects are images under the free module functor $F:\CC_4\rightarrow \DD_4$ (see Lemma \ref{lem:de-gauging}) of objects with trivial twists in $\CC_4$, which we describe as follows.

One can compute, by formulas \eqref{eq:q-dim}, that in addition to the $4$ invertible objects in $\CC_4$, there are $8$ more objects with trivial twists, and they are grouped into the sets $X :=\{2\om_2, 2\om_2+4\om_1, 2\om_2+4\om_3, 2\om_2+4\om_4\}$ 
and 
$Y := \{2\om_1+\om_2+2\om_3,2\om_1+\om_2+2\om_4,\om_2+2\om_3+2\om_4,2\om_1+\om_2+2\om_3+2\om_4\}$.
By \cite[Lemma~2]{Saw02}, the action of the algebra $A_4 \in \CC_4$ on the objects in $X$ and $Y$ (by tensoring) is fixed point free and transitive. In particular, this implies that all of the objects in $X \cup Y$ have trivial double braiding with $A_4$. Therefore, there are only $2$ nontrivial objects with trivial twists in $\DD_4$, which we denote by $Z_1=F(\ld)$ and $Z_2=F(\mu)$ for an arbitrary choice of objects $\ld \in X$ and $\mu \in Y$ (note that $Z_1$ and $Z_2$ does not depend on the choice of $\ld$ and $\mu$).

Moreover, we have 
$$
d_1:=\dim(Z_1)
=28\left(\z_7 + \z_7^6\right) + 14\left(\z_7^2 + \z_7^5\right) + 33
\approx 61.685
$$
and
$$
d_2:=\dim(Z_2)
=
126\left(\z_7 + \z_7^6\right) + 56\left(\z_7^2 + \z_7^5\right) + 157\approx 289.197\,.
$$
Therefore, any nontrivial connected \'etale algebra in $\DD_4$ has to be of the form $L = \1 + a_1Z_1 + a_2Z_2$ for some $a_1$, $a_2 \in \BZ_{\ge 0}$. Moreover, by the local module construction, any connected \'etale algebra $L \in \DD_4$ will give rise to a modular category $(\DD_4)_L^0$ with the property (see \cite[Thm.~4.5]{KO02})
$$
\dim(\DD_4)_L^0 = \frac{\dim(\DD_4)}{\dim(L)^2}\,,
$$
where 
$
\dim(\DD_4) 
= -196
\left[
269(\z_7+\z_7^6)+873(\z_7^2+\z_7^5) + 1357(\z_7^3+\z_7^4)
\right] \approx 489669.5
$.

Since the global dimensions of modular categories are $\ge 1$ (\cite[Thm.~2.3]{ENO05}), upper bounds for $a_1, a_2$ can be given by 
$$
\frac{\dim(\DD_4)}{(1 + a_1d_1)^2} \ge 1
\quad 
\Rightarrow
\quad
a_1 \le 11\,;
\qquad
\frac{\dim(\DD_4)}{(1 + a_2d_2)^2} \ge 1
\quad 
\Rightarrow
\quad
a_2 \le 2\,.
$$ 
Therefore, the set of connected  algebras in $\DD_4$ is a subset of the following set of objects 
$$
E := \{\1 + a_1Z_1 + a_2Z_2\mid 0 \le a_1 \le 11, 0 \le a_2 \le 2\}\,,
$$
with the tensor unit $\1$ corresponding to $a_1 = a_2 = 0$. From now on, we will focus on nontrivial objects, so we assume that $a_1$ and $a_2$ are not both 0.

According to \cite[Lem.~5.3 (c)]{NSW19}, the dimension of any connected \'etale algebra $L \in E$ is totally positive, i.e., all the Galois conjugates of $\dim(L)$ are positive real numbers. This gives more restrictions on $a_1$ and $a_2$. Indeed, using \texttt{SageMath} \cite{sage}, one can see that for any $L \in E$, the Galois conjugates of $\dim(L) = 1 + a_1d_1 + a_2d_2$ are of the form $1+a_1d_1'+a_2d_2'$ and $1 + a_1d''_1 + a_2d''_2$. Here, $d_1'$, $d''_1$ are Galois conjugates of $d_1$ with $d_1' \approx -4.688$ and $d''_1 \approx 0.003$; and $d_2'$, $d''_2$ are Galois conjugates of $d_2$ with $d_2' \approx 0.016$ and $d''_2 \approx -0.213$. Therefore, by direct computation, the only nontrivial $L \in E$ with totally positive dimensions are $L = \1 + Z_2$ and $L = \1 + 2Z_2$. In other words, these are the only two candidates for nontrivial connected \'etale algebras in $\DD_4$.  

Next observe that since $\dim(\DD_4)_L^0$ is an algebraic integer (\cite[Rmk.~2.5]{ENO05}), the algebraic norm of $\frac{\dim(\DD_4)}{\dim(L)^2}$ has to be a (rational) integer. Again using \texttt{SageMath}, we can easily see that only for $L = \1 + 2Z_2$,  $\dim(\DD_4)/\dim(L)^2$ has integral norm.

Finally, since the local module construction preserves pseudounitarity (\cite[Lem.~2.24]{Sch18}), $(\DD_4)_L^0$ is pseudounitary for any connected \'etale algebra $L$. As the FP-dimension of objects are $\ge 1$ (\cite[Prop.~3.3.4]{EGNO}), we have either $\frac{\dim(\DD_4)}{\dim(L)^2} = 1$ or $\frac{\dim(\DD_4)}{\dim(L)^2} \ge 2$. By direct computation, we have
$$
\frac{\dim(\DD_4)}{\dim(\1 + 2Z_2)^2} \textbf{}\approx 1.459\,.
$$
Therefore, $L = \1+2Z_2$ cannot be a connected \'etale algebra.

Combining the above discussions, we have
\begin{thm}
The modular category $\DD_4$ is completely anisotropic, hence $[\DD_4]\in \WW_S$ .\qed
\end{thm}

\bibliographystyle{abbrv}
\bibliography{Reference}

\end{document}